\documentclass[11pt]{amsart}
\usepackage{amssymb}

\newcommand{\dbar}{\ensuremath{\overline\partial}}

\newcommand{\C}{\ensuremath{\mathbb{C}}}
\newcommand{\R}{\ensuremath{\mathbb{R}}}

\makeatletter
\newcommand{\sumprime}{\if@display\sideset{}{'}\sum%
            \else\sum'\fi}
\makeatother

\begin{document}

\numberwithin{equation}{section}

\newtheorem{theorem}{Theorem}[section]
\newtheorem{proposition}[theorem]{Proposition}
\newtheorem{conjecture}[theorem]{Conjecture}
\def\theconjecture{\unskip}
\newtheorem{corollary}[theorem]{Corollary}
\newtheorem{lemma}[theorem]{Lemma}
\newtheorem{observation}[theorem]{Observation}
\theoremstyle{definition}
\newtheorem{definition}{Definition}
\numberwithin{definition}{section}
\newtheorem{remark}{Remark}
\def\theremark{\unskip}
\newtheorem{question}{Question}
\def\thequestion{\unskip}
\newtheorem{example}{Example}
\def\theexample{\unskip}
\newtheorem{problem}{Problem}

\def\vvv{\ensuremath{\mid\!\mid\!\mid}}
\def\intprod{\mathbin{\lr54}}
\def\reals{{\mathbb R}}
\def\integers{{\mathbb Z}}
\def\N{{\mathbb N}}
\def\complex{{\mathbb C}\/}
\def\dist{\operatorname{dist}\,}
\def\spec{\operatorname{spec}\,}
\def\interior{\operatorname{int}\,}
\def\trace{\operatorname{tr}\,}
\def\cl{\operatorname{cl}\,}
\def\essspec{\operatorname{esspec}\,}
\def\range{\operatorname{\mathcal R}\,}
\def\kernel{\operatorname{\mathcal N}\,}
\def\dom{\operatorname{\mathcal D}\,}
\def\linearspan{\operatorname{span}\,}
\def\lip{\operatorname{Lip}\,}
\def\sgn{\operatorname{sgn}\,}
\def\Z{ {\mathbb Z} }
\def\e{\varepsilon}
\def\p{\partial}
\def\rp{{ ^{-1} }}
\def\Re{\operatorname{Re\,} }
\def\Im{\operatorname{Im\,} }
\def\dbarb{\bar\partial_b}
\def\eps{\varepsilon}
\def\O{\Omega}
\def\Lip{\operatorname{Lip\,}}

\def\Hs{{\mathcal H}}
\def\E{{\mathcal E}}
\def\scriptu{{\mathcal U}}
\def\scriptr{{\mathcal R}}
\def\scripta{{\mathcal A}}
\def\scriptc{{\mathcal C}}
\def\scriptd{{\mathcal D}}
\def\scripti{{\mathcal I}}
\def\scriptk{{\mathcal K}}
\def\scripth{{\mathcal H}}
\def\scriptm{{\mathcal M}}
\def\scriptn{{\mathcal N}}
\def\scripte{{\mathcal E}}
\def\scriptt{{\mathcal T}}
\def\scriptr{{\mathcal R}}
\def\scripts{{\mathcal S}}
\def\scriptb{{\mathcal B}}
\def\scriptf{{\mathcal F}}
\def\scriptg{{\mathcal G}}
\def\scriptl{{\mathcal L}}
\def\scripto{{\mathfrak o}}
\def\scriptv{{\mathcal V}}
\def\frakg{{\mathfrak g}}
\def\frakG{{\mathfrak G}}

\def\ov{\overline}

\thanks{Research supported in part by NSF grant DMS-0805852, a U.S.-China Collaboration in Mathematical Research supplementary to NSF grant DMS-0500909, and by Fok Ying Tung Education Foundation Grant 111004.}

\address{Department of Applied Mathematics, Tongji University, Shanghai, 200092, China}
\address{Department of Mathematical Sciences,
Rutgers University, Camden, NJ 08102, U.S.A.}
\email{boychen@tongji.edu.cn}
\email{sfu@camden.rutgers.edu}

\title{Comparison of the Bergman and Szeg\"{o} kernels}
\author{Bo-Yong Chen and Siqi Fu}
\date{}
\maketitle

\bigskip

\begin{abstract}
The quotient of the Szeg\"{o} and Bergman kernels for a smooth bounded pseudoconvex
domains in ${\mathbb C}^n$ is bounded from above by $\delta|\log\delta|^p$ for any $p>n$, where $\delta$ is the distance to the boundary. For a class of domains that
includes  those of D'Angelo finite type and those with plurisubharmonic defining functions, the quotient is also bounded from below by $\delta|\log\delta|^p$ for any $p<-1$. Moreover, for convex domains, the quotient is bounded from above and below by constant multiples of $\delta$.

\bigskip

\noindent{{\sc Mathematics Subject Classification} (2000): 32A25, 32W05, 32U35.}

\smallskip

\noindent{{\sc Keywords}: Bergman kernel, Szeg\"{o} kernel, $\bar\partial$-operator, $L^2$-estimate, Diederich-Forn{\ae}ss exponent, pluricomplex Green function.}
\end{abstract}

\bigskip

\bigskip

\section{Introduction}

The Bergman and Szeg\"o kernels are two important reproducing kernels in complex analysis. They are related yet distinct. Whereas the Bergman kernel $K$ (as a measure) is  biholomorphically invariant, the Szeg\"o kernel $S$ is not. The former is connected to the $\dbar$-problem and the $\dbar$-Neumann Laplacian and the latter the $\bar{\partial}_b$-problem and the Kohn Laplacian. In his book published in 1972, Stein posted the following problem: What are the relations between $K$ and $S$? He further noted that the relation between $K$ and $S$ was known only in very special circumstances (\cite[p.~20]{Stein72}).

There has been an extensive literature that connects the $\dbar$-Neumann Laplacian to boundary pseudo-differential operators associated with the Kohn Laplacian (cf.~\cite{GreinerStein77, NRSW89, Kohn02}) and mapping properties of the Szeg\"{o} projection to that of the Bergman projection (cf.~\cite{Boas87, BonamiCharpentier90, Machedon88, PhongStein77}). However, there are few results, as far as we know, that directly relate these two kernels.

In this paper, we study boundary behavior of the quotient $S(z, z)/K(z, z)$ of the Szeg\"{o} and Bergman kernels for a smooth bounded pseudoconvex domain $\Omega$ in $\C^n$.
When $\Omega$ is strictly pseudoconvex, boundary limiting behavior of the Bergman kernel
$K(z, z)$ was obtained by H\"{o}rmander \cite{Hormander65} and asymptotic expansions for the Bergman and Szeg\"{o} kernels were established by Fefferman \cite{Fefferman74} and
Boutet-Sj\"{o}strand \cite{BoutetSjostrand76}. As a result, $S(z, z)/K(z, z)$ is asymptotically $\delta(z)/n$ near the boundary, where $\delta(z)$ is the Euclidean distance to the boundary. When $\Omega$ is a pseudoconvex domain of finite type in $\C^2$ or a convex domain of finite type in $\C^n$, estimates of the Bergman kernel on diagonal from above and below were obtained by Catlin \cite{Catlin89} and J. Chen \cite{ChenJH89}. Estimates for the Bergman and Szeg\"{o} kernels (on and off diagonal) and their derivatives from above were established by McNeal\cite{McNeal89, McNeal94}, Nagel et al \cite{NRSW89}, and McNeal-Stein \cite{McNealStein97}. It follows that on
these domains, $S(z, z)/K(z, z)\le C\delta(z)$ for some positive constant $C$.

Our main result can be stated as follows:

\begin{theorem}\label{th:main}
Let $\Omega\subset\subset\C^n$ be a pseudoconvex domain with $C^2$-smooth boundary.
\begin{enumerate}
\item For any $a\in (0,\ 1)$, there exists a constant $C>0$ such that
$$
\frac{S(z,z)}{K(z, z)}\le  C \delta(z)|\log(\delta(z))|^{n/a}.
$$

\item If there exist a neighborhood $U$ of $b\Omega$, a bounded continuous plurisubharmonic function $\varphi$ on $U\cap\Omega$, and a defining function $\rho$ of $\Omega$ satisfying $i\partial\dbar\varphi \ge i\rho^{-1} \partial\dbar\rho$ on $U\cap\Omega$ as currents, then there
exist constants $a\in (0,\ 1]$ and $C>0$ such that
$$
\frac{S(z,z)}{K(z, z)}\ge  C \delta(z)|\log(\delta(z))|^{-1/a}.
$$
\end{enumerate}
\end{theorem}

The constant $a$ in the second part of the theorem is a Diederich-Forn{\ae}ss exponent (\cite{DiederichFornaess77}): Namely, there exists a negative plurisubharmonic function $\varphi$ on $\Omega$ such that $C_1 \delta^a(z)\le -\varphi(z)\le C_2 \delta^{a}(z)$ for some positive constants $C_1$ and $C_2$.
It was shown by Catlin \cite{Catlin84, Catlin87} that any smooth bounded pseudoconvex domain of D'Angelo finite type satisfies Property ($P$). Sibony further showed that
for a smooth bounded pseudoconvex domain satisfying Property ($P$), the Diederich-Forn{\ae}ss index, the supremum of the Diederich-Forn{\ae}ss exponents,  is one (see \cite[Theorem~2.4]{Sibony91}).  More recently, Forn{\ae}ss and Herbig \cite{FornaessHerbig08} showed that a smooth bounded domain with a defining function that is plurisubharmonic on the boundary also has Diederich-Forn{\ae}ss index one.

For the convenience of the discussion, a bounded domain that satisfies the condition in (2) will be called {\it $\delta$-regular}. As we will show in Section~\ref{sec:lower}, such a domain is necessarily hyperconvex with a positive Diederich-Forn{\ae}ss index. It is easy to see that the class of $\delta$-regular domains includes smooth bounded pseudoconvex domains with a defining function that is plurisubharmonic on $b\Omega$, and
it is a consequence of the above-mentioned work of Catlin \cite{Catlin87} that this class of domains also includes pseudoconvex domains of D'Angelo finite type (see Proposition~\ref{pro:c} below). Therefore, in light of these and Theorem~\ref{th:main}, we have:

\begin{theorem}
Let $\Omega$ be a smooth bounded pseudoconvex domain in $\C^n$. Suppose that $b\Omega$ is
either of D'Angelo finite type or has a defining function that is plurisubharmonic on $b\Omega$. Then for any constant $a\in (0,\ 1)$, there exist positive
constants $C_1$ and $C_2$ such that
\begin{equation}\label{eq:cor1}
C_1 \delta(z)|\log(\delta(z))|^{-1/a}\le\frac{S(z,z)}{K(z, z)}\le  C_2 \delta(z)|\log(\delta(z))|^{n/a}.
\end{equation}
\end{theorem}

The logarithmic terms in the above theorems do not materialize when the domain is convex. More precisely, we have:

\begin{theorem}\label{th:convex}
Let $\Omega\subset\subset\C^n$ be a bounded convex domain with $C^2$-smooth boundary. Then there exist positive constants $C_1$ and $C_2$ such that
\begin{equation}\label{eq:convex}
C_1 \delta(z)\le S(z,z)/K(z, z)\le  C_2\delta(z).
\end{equation}
\end{theorem}

Our analysis depends on the $L^2$-estimates for the $\bar\partial$-operator by H\"{o}rmander \cite{Hormander65}, Demailly \cite{Demailly82}, and Berndtsson \cite{Berndtsson01}. We also make essential use of Blocki's estimates for the pluricomplex Green function on hyperconvex domains \cite{Blocki04}.

This paper is organized as follows: In Section~\ref{sec:prelim}, we
establish necessary background for the Hardy spaces, the Bergman and Szeg\"{o} kernels. In
Section~\ref{sec:weighted}, we review the relevant $L^2$-estimates of the $\bar\partial$-operator by H\"{o}rmander \cite{Hormander65}, Demailly \cite{Demailly82}, and Berndtsson \cite{Berndtsson01}. The first part of Theorem~\ref{th:main} is proved in Section~\ref{sec:upper} and the second part in Section~\ref{sec:lower}.

\section{Preliminaries}\label{sec:prelim}

We first establish necessary harmonic analysis background. We refer the reader to \cite{Stein72, Stein93} for an extensive treatise on the subject. Let $D$ be a bounded domain in $\R^N$ with $C^2$-smooth boundary.  Let $D_\eps=\{x\in D \mid \delta_D(x)>\eps\}$, where $\delta_D(x)$ denotes the Euclidean distance to the boundary $bD$.  For $1<p<\infty$, the harmonic Hardy space $h^p(D)$ is the space of harmonic functions $f$ such that
\[
\|f\|^p_{h^p}=\limsup_{\eps\to 0^+} \int_{bD_\eps} |f(z)|^p \, dS <\infty.
\]
The level sets $bD_\eps$ in the above definition can be replaced by those of any defining function of $D$ (see \cite{Stein72}). A classical result says that the non-tangential limit $f^*(y)$ of $f$ exists for almost every point $y$ on $bD$. Furthermore, $f^*\in L^p(bD)$, $\|f\|_{h^p}=\|f^*\|_{L^p(bD)}$, and
\[
f(x)=\int_{bD} P(x, y) f^*(y)\, dS(y),
\]
where $P(x, y)$ is the Poisson kernel of $D$.

Throughout the paper, we will use $C$, together with subscripts, to denote positive constants which could be different in different appearances. We will need the following two simple lemmas.

\begin{lemma}\label{lm:hhardy}
Let $D_1\subset D_2$ be bounded domains in $\R^N$ with $C^2$-smooth boundaries. There exists a positive constant $C$ such that
\begin{equation}\label{eq:hhardy}
\|f\|_{h^p(D_1)}\le C\|f\|_{h^p(D_2)}
\end{equation}
for any $f\in h^p(D_2)$.
\end{lemma}

\begin{proof} The proof follows the same lines of arguments as in the proof of Theorem 1 in \cite{Stein72}.  We provide the detail below. Let
\[
g(x)=\int_{bD_2} P_2(x, y)|f^*(y)|^p\, dS(y),
\]
where $P_2(x, y)$ is the Poisson kernel of $D_2$. Then for any $x\in D_2$,
\[
|f(x)|^p=\left|\int_{bD_2} P_2(x, y) f^*(y)\, dS(y)\right|^p\le \int_{bD_2}P_2(x, y)|f^*(y)|^p\, dS(y)=g(x).
\]
Thus $g(x)$ is a harmonic majorant of $|f|^p$ on $D_2$. Now fix a point $x_0$ in $D_1$. Let $G_1(x, y)$ be the Green function of $D_1$.  Let $D_1^\eps=\{x\in D_1 \mid G_1(x, y)>\eps\}$.  Then $-\partial G_1(x, y)/\partial \nu_y  =P_\eps (x, y)$ is the
Poisson kernel of $D_1^\eps$, where $\nu_y$ is the outward normal direction on $bD_1^\eps$.
Let $\pi_\eps\colon bD_1^\eps\to bD_1$ be the projection along the normal direction. Since $P_\eps (x_0, \pi^{-1}_\eps(y))$ converges uniformly on $bD_1$ to $P_1(x_0, y)$ and $C_1=\min\{P_1(x_0, y) \mid y\in bD_1\}>0$, we have
\[
g(x_0)=\int_{bD_1^\eps} P_\eps(x_0, y) g(y)\, dS \ge \frac{C_1}{2} \int_{bD_1^\eps} g(y)\, dS
\]
for sufficiently small $\eps>0$.  It follows that
\[
\begin{aligned}
\int_{b D_1^\eps} |f(x)|^p \, dS &\le \int_{bD_1^\eps} g(x)\, dS \le \frac{2}{C_1} g(x_0)
=\frac{2}{C_1} \int_{b D_2} P_2(x_0, y) |f^*(y)|^p \, dS \\
&\le \frac{2C_2}{C_1} \int_{bD_2} |f^*(y)|^p \, dS =\frac{2C_2}{C_1}\|f\|^p_{h^p(D_2)},
\end{aligned}
\]
where $C_2=\max\{P_2(x_0, y)\mid y\in bD_2\}<\infty$.  Thus \eqref{eq:hhardy} holds with $C=(2C_2/C_1)^{1/p}$.
\end{proof}

In what follows, we will also use $f$ to denote the boundary values $f^*$ for
$f\in h^p(D)$.

\begin{lemma}\label{lm:hardy} Let $D$ be a bounded domain in $\R^N$ with $C^2$-smooth boundary.
For any harmonic function $f$ on $D$,
\begin{equation}\label{eq:hardy1}
\limsup_{\eps\to 0^+} \int_{bD_\eps} |f|^p\, dS
=\limsup_{r\to 1^-} (1-r)\int_D |f(x)|^p \delta^{-r}(x)\, dV.
\end{equation}
Furthermore, when the above limits are finite, then $f\in h^p(D)$ and
\begin{equation}\label{eq:hardy2}
\int_{bD} |f|^p\, dS=\lim_{r\to 1^-} (1-r)\int_D |f(x)|^p \delta^{-r}(x)\, dV.
\end{equation}
\end{lemma}

\begin{proof}  If the limit on the left hand side of \eqref{eq:hardy1} is finite,  then $f\in h^p(D)$.
Hence
\[
\lim_{\eps\to 0^+} \int_{bD_\eps} |f|^p\, dS=\int_{bD} |f|^p\, dS.
\]
Write
\[
\lambda(\eps)=\int_{bD_\eps} |f|^p\, dS.
\]
Then $\lambda(\eps)$ is continuous on $[0, a]$ for any sufficiently small $a>0$. Therefore,
\[
\lim_{r\to 1^-} (1-r)\int_{D} |f|^p \delta^{-r}\, dV=\lim_{r\to 1^-} (1-r)\int_0^{a}\eps^{-r}\lambda(\eps)\, d\eps=\lambda(0)=\int_{bD} |f|^p\, dS.
\]
Now suppose the limit on the right hand side of \eqref{eq:hardy1} is finite. For any sufficiently small $0<\eps_1<\eps_2$, we assume that $\lambda(\eps)$ takes its minimum on $[\eps_1, \eps_2]$ at $\eps_0$. Then
\[
(1-r)\int_{D} |f|^p \delta^{-r}\, dV\ge (1-r)\int_{\eps_1\le\delta\le\eps_2} |f|^p \delta^{-r}\, dV\ge(\eps_2^{1-r}-\eps_1^{1-r}) \lambda(\eps_0).
\]
Taking $\liminf_{\eps_1\to 0^+}$ and then $\limsup_{r\to 1^-}$, we then have
\[
\infty>\limsup_{r\to 1^-} (1-r)\int_D |f(x)|^p \delta^{-r}(x)\, dV\ge \liminf_{\eps_1\to 0^+} \lambda(\eps_0).
\]
It follows that there exists a sequence $\eps_j\to 0^+$ such that $\lambda(\eps_j)$ is bounded. Let $\pi_\eps\colon bD_\eps\to bD$ be the projection along the outward
normal direction. Then $f_j(x)=f(\pi^{-1}_{\eps_j}(x))$ is a bounded sequence in $L^p(bD)$. By Alaoglu's theorem, it has a subsequence that converges to some $\tilde f\in L^p(bD)$ in the weak* topology. It follows that
\[
f(x)=\int_{bD} P(x, y)\tilde f(y)\, dS(y).
\]
Hence $f\in h^p(D)$ and we can refer back to the first part of the proof.
\end{proof}

We now review the rudiments on the Bergman and Szeg\"{o} kernels. Let $\Omega$ be a bounded domain in $\C^n$ and let $A^2(\Omega)$ be the Bergman space, the space of square integrable holomorphic functions on $\Omega$. The Bergman kernel $K_\Omega(z, w)$ is the reproducing kernel of $A^2(\Omega)$:
\[
f(z)=\int_\Omega K_\Omega(z, w) f(w)\, dV, \qquad \forall f\in A^2(\Omega), \ \forall z\in\Omega.
\]
Assume that $b\Omega$ is of class $C^2$. The Hardy space $H^2(\Omega)$ is the space of holomorphic functions on $\Omega$ that are also in  $h^2(\Omega)$.  The Szeg\"{o} kernel is
the reproducing kernel of $H^2(\Omega)$:
\[
f(z)=\int_{b\Omega} S_\Omega(z, w) f(w)\, dS, \qquad \forall f\in H^2(\Omega), \ \forall z\in\Omega.
\]
It follows from these reproducing properties that
\begin{equation}\label{eq:bergman}
K_\Omega(z, z)=\sup\{ |f(z)|^2; \  f\in A^2(\Omega), \|f\|_\Omega\le 1\}
\end{equation}
and
\begin{equation}\label{eq:szego}
\quad S_\Omega(z, z)=\sup\{|f(z)|^2; \  f\in H^2(\Omega), \|f\|_{b\Omega}\le 1\}.
\end{equation}
From \eqref{eq:bergman}, we know that the Bergman kernel has the decreasing property: if $\Omega_1\subset\Omega_2$, then $K_{\Omega_1}(z, z)\ge K_{\Omega_2}(z, z)$.  Combining Lemma~\ref{lm:hhardy} with \eqref{eq:szego}, we have:

\begin{lemma}\label{lm:lszego}
Let $\Omega_1\subset\Omega_2$ be bounded domains in $\C^n$ with $C^2$-smooth boundaries. Then there exists a constant $C>0$ such that
\begin{equation}\label{eq:lszego}
S_{\Omega_2}(z, z)\le C S_{\Omega_1}(z, z)
\end{equation}
for all $z\in\Omega_1$.
\end{lemma}

\section{Weighted $L^2$-estimates for the $\bar\partial$-operator}\label{sec:weighted}

In this section, we review relevant weighted $L^2$-estimates for the $\dbar$-operator of H\"{o}rmander, Demailly,  and Berndtsson. We will only state their results for $(0, 1)$-forms, which are what we will need later in this paper. Let $\Omega$ be a bounded pseudoconvex domain in $\C^n$ and let $\psi$ be a plurisubharmonic function on $\Omega$. Let $L^2(\Omega,e^{-\psi})$ be the Hilbert space of all measurable functions satisfying
$$
\|f\|^2_\psi=\int_\Omega |f|^2 e^{-\psi}\, dV<\infty
$$
and let $L^2_{(0, 1)}(\Omega, e^{-\psi})$ be the space of $(0, 1)$-forms with coefficients in $L^2(\Omega, e^{-\psi})$.  Suppose $\partial\dbar\psi \ge c \partial\dbar |z|^2$ as currents where $c$ is a positive continuous function. (Here and in what follows, we will drop the letter $i$ from the real $(1, 1)$-form $i\partial\dbar\psi$.) H\"ormander's theorem says that for any $\dbar$-closed $(0, 1)$-form $f$, one can solve the equation
\begin{equation}\label{eq:d-bar}
\bar{\partial}u=f
\end{equation}
in the sense of distribution, together with the estimate
\begin{equation}\label{eq:hormander}
\int_\Omega |u|^2 e^{-\psi}\, dV\le 2\int_\Omega |f|^2 e^{-\psi}/c\, dV,
\end{equation}
provided the right hand side is finite (\cite[Theorem 2.2.1$'$ ]{Hormander65}; see also \cite[Lemma4.4.1]{Hormander91}).  Suppose $\psi\in C^2(\Omega)$. For any $(0, 1)$-form $f$, let
\[
|f|_{\partial\dbar\psi}=\sup\{|\langle f, X\rangle|; \ \ X\in T^{0, 1}(\Omega),\ |X|_{\partial\dbar\psi} \le 1\}
\]
be the norm induced by the $(1, 1)$-form $\partial\dbar\psi$, where $\langle \cdot, \cdot\rangle$
denotes the pairing of a form and a vector and $|X|_{\partial\dbar\psi}=\partial\dbar\psi(\overline{X}, X)$ is the length of $X$ with respect to $\partial\dbar\psi$. According to Demailly's reformulation of H\"{o}rmander's theorem, one can solve $\dbar u=f$ with the following estimate

\begin{equation}\label{eq:demailly1}
\int_\Omega |u|^2 e^{-\psi}\, dV\le \int_\Omega |f|_{\partial\bar{\partial}\psi}^2 e^{-\psi}\, dV,
\end{equation}
provided the right hand side is finite (see \cite[Theorem~4.1]{Demailly82}). It follows that for any function $u$ in the orthogonal complement of the nullspace $\scriptn(\dbar)$
in $L^2(\Omega, e^{-\psi})$, we have
\begin{equation}\label{eq:demailly2}
\int_\Omega |u|^2 e^{-\psi}\, dV\le \int_\Omega |\dbar u|_{\partial\bar{\partial}\psi}^2 e^{-\psi}\, dV.
\end{equation}

The following theorem is a slight reformulation of a result due to Berndtsson (\cite[Theorem~2.8]{Berndtsson01}). Berndtsson's proof uses an integration by parts formula related to the $\partial\dbar$-Bochner-Kodaira technique of Siu (see Section 3 in \cite{Siu82}). We provide a proof here as a simple application of \eqref{eq:demailly2}.
Similar approach was used in \cite{BerndtssonCharpentier00} to prove an estimate of Donnelly-Fefferman \cite{DonnellyFefferman83}.

\begin{theorem}\label{prop:bo}
Let $\Omega\subset\subset\C^n$ be a bounded pseudoconvex domain.  Let $\rho\in C^2(\Omega)$ with $\rho<0$.  Suppose that there exists a plurisubhamornic function $\psi\in C^2(\Omega)$ such that
$$
\Theta:=(-\rho)\partial\bar{\partial} \psi+\partial\bar{\partial}\rho
$$
is positive. Let $u$ be the solution to \eqref{eq:d-bar} that is orthogonal to $\scriptn(\dbar)$ in $L^2(\Omega, e^{-\psi})$. Then for any $0<r<1$,
\begin{equation}\label{eq:bo}
(1-r)\int_\Omega |u|^2 (-\rho)^{-r}e^{-\psi}\, dV\le \frac1r \int_\Omega |f|^2_{\Theta}(-\rho)^{1-r} e^{-\psi}\, dV.
\end{equation}
\end{theorem}

\begin{proof} Let $\phi=-r\log(-\rho)$ and $\varphi=\phi+\psi$. Then $ue^\phi \perp \scriptn(\dbar)$ in $L^2(\Omega, e^{-\varphi})$.  Applying \eqref{eq:demailly2} to
$ue^{\phi}$ with weight $e^{-\varphi}$, we have
\[
\int_\Omega |u|^2 e^{\phi-\psi} \, dV\le \int_\Omega |\dbar u+ u\dbar\phi|^2_{\partial\dbar\varphi} e^{\phi-\psi}\, dV.
\]
It remains to show that
\begin{equation}\label{eq:d1}
|\dbar u+u\dbar\phi|^2_{\partial\dbar\varphi}\le r|u|^2+\frac{1}{r}|\dbar u|^2_{\Theta}(-\rho).
\end{equation}
Notice that
\[
\begin{aligned}
\partial\dbar\varphi&=\frac{r}{-\rho}\Theta+\frac{1}{r}\partial\phi\wedge\dbar\phi+
(1-r)\partial\dbar\psi \\
&\ge \frac{r}{-\rho}\Theta+\frac{1}{r}\partial\phi\wedge\dbar\phi =: \widetilde\Theta.
\end{aligned}
\]
Thus
\begin{equation}\label{eq:d2}
|\dbar u+u\dbar\phi |^2_{\partial\dbar\varphi}\le |\dbar u+u\dbar\phi|^2_{\widetilde\Theta}
=\sup\{\frac{|\langle\dbar u+u\dbar\phi, X\rangle|^2}{\frac{r}{-\rho}|X|^2_\Theta+\frac{1}{r}|\langle\dbar\phi, X\rangle|^2};\ \  X\in T^{0, 1}(\Omega)\}.
\end{equation}
Inequality \eqref{eq:d1} then follows from \eqref{eq:d2} and the inequalities $|\langle \dbar u, X\rangle|\le |\dbar u|_\Theta |X|_\Theta$ and
\[
2|u\langle \dbar\phi, X\rangle\overline{\langle\dbar u, X\rangle}|\le 2|u||X|_\Theta |\dbar u|_\Theta |\langle \dbar\phi, X\rangle|\le\frac{r^2}{-\rho}|u|^2|X|^2_\Theta+\frac{-\rho}{r^2}|\dbar u|^2_\Theta|\langle\dbar \phi, X\rangle|^2.
\]

\end{proof}

\section{Upper bound estimates}\label{sec:upper}

We prove the first part of Theorem~\ref{th:main} in this section. For a bounded domain $\Omega$ in $\C^n$, the {\it pluricomplex Green function} with a pole at $w\in \Omega$ is defined by
$$
g_\Omega(z,w)=\sup\left\{u(z); \ u\in PSH(\Omega),\,u<0,\,\lim\sup_{z\rightarrow w}(u(z)-\log|z-w|)<\infty\right\}.
$$
It is known that for any bounded hyperconvex domain $\Omega$, the pluricomplex Green function $g_\Omega(\cdot, w)\colon \Omega\to [-\infty, 0)$ is a continuous plurisubharmonic function such that $\lim_{z\to b\Omega} g_\Omega(z, w)=0$ (\cite{Demailly87}; see also Chapter 5 in \cite{Klimek91}).

Recall that a constant $a\in (0,\, 1]$ is said to be a {\it Diederich-Forn{\ae}ss exponent} for a bounded pseudoconvex domain $\Omega$ if there exist a negative plurisubharmonic function $\varphi$ on $\Omega$ and positive constants $C_1$ and $C_2$ such that
\begin{equation}\label{eq:df}
C_1 \delta^a(z)\le -\varphi(z) \le C_2 \delta^a(z).
\end{equation}
The supremum of all Diederich-Forn{\ae}ss exponents is called the {\it Diederich-Forn{\ae}ss index} of $\Omega$. It follows from the work of Diederich and Forn{\ae}ss that any bounded pseudoconvex domain with $C^2$-smooth boundary has a positive
Diederich-Forn{\ae}ss index, which can be arbitrarily small (\cite{DiederichFornaess77,
DiederichFornaess77b}). It was proved by Demailly \cite{Demailly87} that  bounded pseudoconvex Lipschitz domains are hyperconvex. More recently, Harrington \cite{Harrington07} showed that bounded pseudoconvex Lipchitz domains have indeed
positive Diederich-Forn{\ae}ss indices.

We will make essential use of the following quantitative estimate for the pluricomplex Green function due to Blocki (\cite[Theorem~5.2]{Blocki04};  see also \cite{Herbort00} for prior related results). We provide a proof below, following mostly Blocki's arguments\footnote{There are slight inaccuracies in the proof of Theorem 5.2 in \cite{Blocki04}: The inequality (5.6) and the choice of $\eps$ there seem to be incorrect.},
because we will need to refer back to it.

\begin{theorem}\label{th:blocki}
Let $\Omega$ be a bounded pseudoconvex domain in ${\mathbb C}^n$. Suppose there exists a negative plurisubharmonic function $\varphi$ on $\Omega$ such that
\begin{equation}\label{eq:est}
C_1\delta^a(z)\le -\varphi(z)\le C_2\delta^b(z), \qquad z\in\Omega
\end{equation}
for some positive constants $C_1, C_2$, and $a\ge b$. Then there exists positive constants $\delta_0$ and $C$ such that
\begin{equation}\label{eq:blocki}
\{z\in\Omega;\ g_\Omega(z,w)\le -1\}\subset \left\{ C^{-1}\delta^{\frac{a}b}(w)|\log{\delta(w)}|^{-\frac1b}\le \delta(z)\le C\delta^{\frac{b}a} (w)|\log{\delta(w)}|^{\frac{n}a}\right\},
\end{equation}
for any $w\in \Omega$ with $\delta(w)\le \delta_0$.
\end{theorem}

\begin{proof} Assume that
$\Omega$ has diameter $R$. Let $w\in\Omega$ with $r=\delta(w)\le e^{-2}$. Let $z\in\Omega$. Suppose that $\delta=\delta(z)\le r/2$. It follows from comparison with the pluricomplex Green function of $B(w, R)$ that
\begin{equation}\label{eq:est1}
g_{\Omega}(\zeta, w)\ge \log(|\zeta-w|/R)
\end{equation}
for all $\zeta\in\Omega$. By the maximal property of the pluricomplex Green function, we have
\begin{equation}\label{eq:fu1}
g_\Omega(\zeta, w)\ge \frac{\log(2R/r)}{\inf\{|\varphi(\zeta)|; \ \zeta\in \overline{B(w, \frac{r}2)}\}} \varphi(\zeta)
\end{equation}
on $\Omega\setminus B(w, r/2)$ because the same inequality holds on the boundary.  By \eqref{eq:est},
\begin{equation}\label{eq:est2}
\inf\{|\varphi(\zeta)|; \ \zeta\in \overline{B(w, \frac{r}2)}\}\ge C (r/2)^a.
\end{equation}
Therefore,
\begin{equation}\label{eq:b1}
g_\Omega(z, w)\ge -C\frac{\delta^b}{r^a}\log\frac{1}{r}.
\end{equation}
It follows that
\begin{align}
\{z\in\Omega; \ g_\Omega(z, w)\le -1\}&\subset\left\{\delta(z)> \frac{r}2 \text{ or }
\delta(z)\ge C r^{\frac{a}b}(\log(1/r))^{-\frac1b} \right\}\notag\\
&\subset \left\{\delta(z)\ge C^{-1}\delta^{\frac{a}b}(w)|\log{\delta(w)}|^{-\frac1b}\right\},\label{eq:bb}
\end{align}
provided the last constant $C$ is sufficiently large.

Now suppose that $e^{-2}\ge\delta(z)\ge 2r$. It follows from \eqref{eq:b1} that for any $0<\eps<r/2$,
\begin{equation}\label{eq:fu2}
\inf_{\delta(\zeta)=\eps} g_\Omega(\zeta, w) \ge -C\frac{\eps^b}{r^a} \log\frac{1}r.
\end{equation}
We also obtain from \eqref{eq:b1} that
\begin{equation}\label{eq:fu3}
g_\Omega(w, z)\ge -C \frac{r^b}{\delta^a}\log \frac{1}{\delta},
\end{equation}
by reversing the r\^{o}les of $z$ and $w$. By Theorem 5.1 in \cite{Blocki04}, we have
\begin{equation}\label{eq:b2}
g_\Omega(z, w)\ge -C \frac{\log\frac{1}{\eps}}{\log\frac{r}{2\eps}}\left(\frac{\eps^b}{r^a}\log\frac{1}{r}+
\frac{r^{\frac{b}{n}}}{\delta^{\frac{a}{n}}}
(\log\frac1\eps)^{1-\frac{1}{n}}(\log\frac{1}{\delta})^{\frac{1}{n}}\right).
\end{equation}
We have followed closely Blocki's proof thus far. Here is where we start to deviate. Suppose further that
\begin{equation}\label{eq:est3}
\delta(z)\ge r^{\frac{b}{a}}\left(\log\frac{1}{r}\right)^{\frac{n}{a}}.
\end{equation}
Set
\begin{equation}\label{eq:est4}
\eps=\frac{1}{2}\frac{r^{\frac{1}{n}+\frac{a}{b}}}{\delta^{\frac{a}{bn}}}
\left(\frac{\log\log\frac{1}{r}}{\log\frac{1}{r}}\right)^\frac{1}{b}.
\end{equation}
Since $\delta\ge 2r$, $\log\log\frac{1}{r}\le \log\frac{1}{r}$, and
\begin{equation}\label{eq:est5}
\frac{1}{n}+\frac{a}{b}-\frac{a}{bn}\ge 1,
\end{equation}
we have
\[
\eps\le \frac{r^{\frac{1}{n}+\frac{a}{b}-\frac{a}{bn}}}{2^{1+\frac{a}{bn}}}<\frac{1}{2} r
\]
as required.
From \eqref{eq:est3} and \eqref{eq:est4}, we know that
\[
\frac{r}{2\eps}\ge r^{1-\frac{a}{b}} \left(\log\frac{1}{r}\right)^\frac{1}{b}\left(\frac{\log\frac{1}{r}}{\log\log\frac{1}{r}}
\right)^\frac{1}{b}\ge C \left(\log\frac{1}{r}\right)^\frac{1}{b},
\]
provided $r=\delta(w)$ is sufficiently small. (Recall that $b\le a$.)
Therefore,
\begin{equation}\label{eq:est6}
\log\frac{r}{2\eps}\ge C\log\log\frac{1}{r}.
\end{equation}
It is easy to see that
\begin{equation}\label{eq:est7}
\log\frac{1}{\eps}\le C \log \frac{1}{r} \quad\text{and}\quad
\log\frac{1}{\delta}\le C\log\frac{1}{r}.
\end{equation}
Combining \eqref{eq:est4}, \eqref{eq:est6}, and \eqref{eq:est7} with \eqref{eq:b2}, we then obtain
\[
g_\Omega(z, w)\ge -C\frac{r^{\frac{b}{n}}}{\delta^{\frac{a}{n}}} \log\frac{1}{r}.
\]
Therefore,
\[
\{g_\Omega(z, w)<-1\}\subset\{z\in\Omega; \ \delta(z)\ge e^{-2} \text{ or } \delta(z)\le C\delta^{b/a}(w)|\log \delta(w)|^{n/a}\}.
\]
Together with \eqref{eq:bb}, we then obtain \eqref{eq:blocki} by choosing a sufficiently small $\delta_0$ and a sufficiently large $C$. \end{proof}

We also need the following localization of the Bergman kernel (\cite[Lemma~4.2]{Chen99}; also \cite[Proposition~3.6]{Herbort99}).

\begin{proposition}\label{prop:c}
Let $\Omega$ be a bounded pseudoconvex domain in $\C^n$. Then there exists a positive constant $C$ such that for any $w\in \Omega$,
\[
K_\Omega(w,w)\ge C K_{\{g_\Omega(\cdot,w)< -1\}}(w,w).
\]
\end{proposition}

To illustrate the idea of the proof, we first prove the following weaker version of Theorem~\ref{th:main}~(1):

\begin{proposition}\label{prop:weak}
Let $\Omega\subset\subset\C^n$ be a pseudoconvex domain with $C^2$-smooth boundary. Suppose that the Diederich-Forn{\ae}ss index of $\Omega$ is $\beta$.
Then for any $a\in (0,\ \beta)$, there exists a constant $C>0$ such that
\begin{equation}\label{eq:weak}
\frac{S(z,z)}{K(z, z)}\le  C \delta(z)|\log(\delta(z))|^{n/a}.
\end{equation}
\end{proposition}

\begin{proof} By the definition of the Diederich-Forn{\ae}ss index, there exists a negative plurisubharmonic function $\varphi$ satisfying \eqref{eq:df}.  By Theorem~\ref{th:blocki}, there exists a positive constant $C$ such that
\begin{equation}\label{eq:b7}
\{g_\Omega(\cdot, w)<-1\}\subset \{\delta(\cdot)<C\delta(w)|\log\delta(w)|^{n/a}\}
\end{equation}
for any $w\in\Omega$ sufficiently closed to the boundary.  Therefore, for any $f\in H^2(\Omega)$,
\[
\int_{\{g_\Omega(\cdot, w)<-1\}} |f|^2\, dV\le \int_{0}^{C\delta(w)|\log\delta(w)|^{n/a}} \, d\eps
\int_{b\Omega_\eps} |f|^2 \, dS\le C\|f\|^2_{b\Omega}\delta(w)|\log\delta(w)|^{n/a}.
\]
It then follows from the extremal properties \eqref{eq:bergman} and \eqref{eq:szego} that
\[
S(w, w)\le C \delta(w)|\log\delta(w)|^{n/a}K_{\{g(\cdot, w)<-1\}}(w, w).
\]
Applying Proposition~\ref{prop:c}, we then conclude the proof of the proposition.
\end{proof}

To get from Proposition~\ref{prop:weak} to the first statement of Theorem~\ref{th:main}, we use Lemma~\ref{lm:lszego} to localize the Szeg\"{o} kernel and then apply the following fact: For any $z_0\in b\Omega$ and $a\in (0,\ 1)$, there exist a defining function $r$ of $\Omega$ and a neighborhood $U$ of $z_0$ such that $\varphi_2=-(-r)^a$ is strictly plurisubharmonic on $U\cap\Omega$ (\cite{DiederichFornaess77}, Remark on p.~133).  The problem is that this function $\varphi_2$ is not an exhaustion function of $U\cap\Omega$ and thus one cannot directly apply Theorem~\ref{th:blocki}. We now show how to overcome this difficulty and prove Theorem~\ref{th:main}~(1).

Let $\chi(t)$ be a smooth function such that $\chi(t)=0$ when $t\le 1$, $\chi(t)>0$ is strictly increasing and convex when $t>1$. We may further assume that $\chi(t)=\exp(-1/(t-1))$ when $t\in (1,\ 5/4)$ so that $(\chi(t))^b\in C^\infty(\R)$ for any positive number $b$. Let
\[
\varphi_1(z)=-(-r)^a+M\chi(|z-z_0|^2/m^2).
\]
Then $\varphi_1$ is strictly plurisubharmonic on $U\cap\Omega$.  Write $g(z)=M\chi(|z-z_0|^2/m^2)$. Let
\[
\widetilde\Omega=\{z\in\Omega; \ \ \varphi_1(z)=-(-r)^a+g<0\}.
\]
By choosing $m$ sufficiently small and $M$ sufficiently large, we know that
$B(z_0, m)\cap \Omega\subset\widetilde\Omega$ and
$\widetilde\Omega\subset B(z_0, 2m)$. Furthermore, $\widetilde\Omega$ is pseudoconvex with a $C^2$-smooth defining function
\[
\tilde r=r+g^{1/a}
\]
(see, for example, \cite[pp.~470--471]{Bell86}).

Evidently, $\varphi_1$ is a plurisubharmonic exhaustion function for $\widetilde\Omega$. However, $\varphi_1$ does not satisfy \eqref{eq:df}. In fact, it is easy to show that
there exists a positive constant $C_1$ such that
\begin{equation}\label{eq:loc1}
C_1|\tilde r|\le -\varphi_1\le |\tilde r|^a, \quad z\in\widetilde\Omega.
\end{equation}
If we directly invoke Theorem~\ref{th:blocki} with \eqref{eq:loc1}, we obtain
\[
\{g_{\widetilde\Omega}(\cdot, w)<-1\}\subset \{\delta(\cdot)<C\delta^a(w)|\log\delta(w)|^{n}\}.
\]
Consequently, we have as in the proof of Proposition~\ref{prop:weak} that
\[
S(z, z)/K(z, z)\le C \delta^{a}(z)|\log\delta(z)|^{n},
\]
which is even weaker than \eqref{eq:weak}.

Instead of directly appealing to Theorem~\ref{th:blocki}, we proceed
as follows. We follow the proof of Theorem~\ref{th:blocki} with $\Omega$ replaced by $\widetilde\Omega$, and with $\delta(z)=\delta_{\widetilde\Omega}(z)$ now denoting the Euclidean distance to $b\widetilde\Omega$. Notice that $C^{-1}\delta\le |\tilde r|\le C\delta$ on $\widetilde\Omega$. Assume that $|w-z_0|<m$. Applying \eqref{eq:fu1} to the function $\varphi_1$, we have
\begin{equation}\label{eq:fu2b}
\inf_{\delta(\zeta)=\eps} g_{\widetilde\Omega}(\zeta, w)\ge \frac{\log(2R/r)}{\inf\{|\varphi_1(\zeta)|; \ \zeta\in \overline{B(w, \frac{r}2)}\}} \inf_{\delta(\zeta)=\eps}\varphi_1(\zeta)\ge -C\frac{\eps^a}{r^a} \log\frac{1}r,
\end{equation}
which is our analogue in this case to \eqref{eq:fu2}. (Here we have $a=b$.)  Now applying \eqref{eq:fu1} to the function $\varphi_2=-(-r)^a$ with the r\^{o}le of $z$ and $w$ reversed, we have
\[
g_{\widetilde\Omega}(\zeta, z)\ge \frac{\log(2R/\delta)}{\inf\{|\varphi_2(\zeta)|; \ \zeta\in \overline{B(z, \frac{\delta}2)}\}} \varphi_2(\zeta).
\]
on $\widetilde\Omega\setminus B(z, \delta/2)$. It follows that
\begin{equation}\label{eq:fu3b}
g_{\widetilde\Omega}(w, z)\ge -C\frac{r^a}{\delta^a}\log\frac{1}{\delta},
\end{equation}
which plays the r\^{o}le of \eqref{eq:fu3} in this case. Following exactly
the same lines for the rest of the proof of Theorem~\ref{th:blocki}, we then
obtain
\[
\{z\in\widetilde\Omega; \ g_{\widetilde\Omega}(z, w)<-1\}\subset \{z\in\widetilde\Omega; \ \delta(z)\le C\delta(w)|\log\delta(w)|^{n/a}\}.
\]
From the proof of Proposition~\ref{prop:weak}, we then have
\[
\frac{S_{\widetilde\Omega}(w,w)}{K_{\widetilde\Omega}(w, w)}\le  C \delta(w)|\log(\delta(w))|^{n/a},
\]
when $w$ is sufficiently close to $z_0$.  By the localization property of the Bergman kernel (see the proof of Theorem 1 in \cite{Ohsawa81}; also \cite[Proposition 1]{DFH84}), $K_\Omega(w, w)\ge CK_{\widetilde\Omega}(w, w)$. Together with Lemma~\ref{lm:lszego}, we then conclude the proof of the first statement in Theorem~\ref{th:main}.

\section{Lower bound estimates}\label{sec:lower}

Recall that a continuous function $\rho$ is said to be a defining function of a domain $\Omega\subset\C^n$ if $\Omega=\{z\in \C^n; \ \rho(z)<0\}$ and $C^{-1}\delta\le\rho\le C\delta$ for a constant $C>0$. We also assume the defining function $\rho$ to be in the same smoothness class as that of the boundary $b\Omega$. A bounded domain $\Omega\subset\C^n$ is {\it $\delta$-regular} if there exist a neighborhood $U$ of $b\Omega$, a bounded continuous plurisubharmonic function $\varphi$ on $U\cap\Omega$, and a defining function $\rho$ of
$\Omega$ such that
\begin{equation}\label{eq:delta}
\partial\dbar\varphi \ge \rho^{-1} \partial\dbar\rho
\end{equation}
on $U\cap\Omega$ as currents. By adding $|z|^2$ to $\varphi$, we may assume that it is strictly plurisubharmonic. By Richberg's approximation theorem (\cite[Satz~4.3]{Richberg68}), we may further assume that $\varphi\in C^\infty(\Omega)$.

\begin{proposition}\label{prop:h} Let $\Omega\subset\subset\C^n$ be a $\delta$-regular domain. Then $\Omega$ is hyperconvex with a positive Diederich-Forn{\ae}ss index.
\end{proposition}

\begin{proof} Let $\varphi$ and $\rho$ be the functions that satisfy \eqref{eq:delta}. Assume $0\le\varphi\le M$ for some positive constant $M$. Let $\psi=e^\varphi$ and $K=e^M$. Then
\begin{equation}\label{eq:s}
1\le\psi\le K,  \quad \partial\psi\wedge\dbar\psi\le K \partial\dbar\psi, \quad
\text{ and } \quad \partial\dbar\psi\ge \frac{\partial\dbar\rho}{\rho}+K^{-1}\partial\psi\wedge\dbar\psi,
\end{equation}
on $U\cap\Omega$. Let
\[
\tilde\rho=\rho e^{-\psi} \quad \text{and} \quad r=-(-\tilde\rho)^\eta.
\]
It follows from a simple (formal) computation and \eqref{eq:s} that
\[
\begin{aligned}
\partial\dbar r=&\eta(-\tilde\rho)^\eta\left(\partial\dbar -\log(-\tilde\rho)-\eta\frac{\partial\tilde\rho\wedge\dbar\tilde\rho}{\tilde\rho^2}\right)\\
&\ge \eta(-\tilde\rho)^\eta\left(\frac{1}{K}\partial\psi\wedge\dbar\psi+\frac{\partial\rho
\wedge\dbar\rho}{\rho^2}-
\eta\frac{\partial\tilde\rho\wedge\dbar\tilde\rho}{\tilde\rho^2}\right)\\
&\ge\eta(-\tilde\rho)^\eta\left((1-\eta)\frac{\partial\rho
\wedge\dbar\rho}{\rho^2}+\eta\frac{\partial\rho}{\rho}\wedge\dbar\psi+\eta\partial\psi\wedge
\frac{\dbar\rho}{\rho}+(\frac{1}{K}-\eta)\partial\psi\wedge\dbar\psi\right).
\end{aligned}
\]
We then obtain from the Schwarz inequality that $\partial\dbar r$ is a positive current on $U\cap\Omega$, provided $\eta$ is sufficiently small. The extension of $r$ to the whole domain $\Omega$ is standard (see \cite[p.~133]{DiederichFornaess77}).
\end{proof}

\begin{proposition}\label{pro:c}
Let $\Omega$ be a smooth pseudoconvex bounded domain in $\C^n$. If $\Omega$ has a defining
function that is plurisubharmonic on $b\Omega$ or if $b\Omega$ is of D'Angelo finite type, then $\Omega$ is $\delta$-regular.
\end{proposition}

\begin{proof} If $\Omega$ has a defining function $\rho$ that is plurisubharmonic on $b\Omega$. Then
$\partial\dbar\rho\ge  C\rho \partial\dbar |z|^2$. Therefore in this case, we
can choose $\varphi(z)=C|z|^2$ with a sufficiently large $C$.

The case when $\Omega$ is of D'Angelo finite type is a consequence of Catlin's construction of bounded plurisubharmonic function \cite{Catlin87} (see \cite[p.~464]{Straube97} for a related discussion): There exist positive constants $\tau<1$, $C>0$, and a smooth bounded plurisubharmonic function $\lambda$ on $\Omega$ such that
\begin{equation}
\partial\bar{\partial}\lambda\ge C \frac{\partial\bar{\partial}|z|^2}{|\rho|^{\tau}}.
\end{equation}
By Oka's lemma, we can choose a defining function $\rho$ such that $\partial\dbar(-\log(-\rho))\ge \partial\dbar |z|^2$. It then follows from a theorem
of Diederich-Fornaess that for any sufficiently small $\eta$,
\begin{equation}
\partial\bar{\partial}( -(-\rho)^\eta )\ge C \eta |\rho|^{\eta}\left(\partial\dbar |z|^2+|\rho|^{-2}\partial\rho\wedge\dbar\rho\right)
\end{equation}
for some positive constant $C$ (\cite{DiederichFornaess77}, Theorem 1 and its proof; compare also \cite[Lemma 2.2]{CSW04}).

Now we fix an $\eta\in (0,\ \tau)$. Write $N=|\partial\rho|^{-1}\sum \rho_{\bar z_j}{\partial/\partial z_j}$.
For any $(1, 0)$-vector $X$, write $X_N=\langle X, N\rangle N$ and
$X_T=X-X_N$.  By the pseudoconvexity of $b\Omega$, we have
\[
\rho^{-1}\partial\dbar\rho (X, \ov{X}) \le C(|X|^2+|\rho|^{-1}|X||X_N|)\le C(|\rho|^{-\eta} |X|^2+ |\rho|^{-2+\eta}|X_N|^2)
\]
The desirable function is then given by
\[
\varphi=C(\lambda-(-\rho)^\eta)
\]
for any sufficiently large $C>0$.
\end{proof}

We are now in position to prove the second part of Theorem~\ref{th:main}.  Let $\kappa$ be
a standard Friedrichs mollifier. Let $\eps_j$ be a decreasing sequence of positive number tending to $0$.  Let $w$ be a point in $\Omega$, sufficiently closed to the boundary $b\Omega$. Let $g_j=g_\Omega(\cdot, w)*\kappa_{\eps_j}$. Then $g_j$ is a decreasing sequence of plurisubharmonic functions on $\Omega_j=\{z\in\Omega; \ \delta(z)<\eps_j\}$ with limit $g_\Omega(\cdot, w)$. By Oka's lemma, $\Omega_j$ is pseudoconvex.  Let
$$
\psi=2ng_\Omega(\cdot, w)-\log(-g_\Omega(\cdot, w)+1)+\varphi  \quad \text{and}\quad
\psi_j=2n g_j-\log(-g_j+1)+\varphi,
$$
where $\varphi$ is the smooth bounded strictly plurisubharmonic function, obtained from the $\delta$-regularity assumption, such that $\partial\dbar\varphi \ge \rho^{-1}\partial\dbar\rho$ for a defining function $\rho$ of $\Omega$. Clearly $\psi_j$ is a plurisubharmonic function on $\Omega_j$. Moreover,
\begin{equation}\label{eq:f1}
\partial\bar{\partial}\psi_j\ge \partial\log(-g_j+1)\wedge\bar{\partial}\log(-g_j+1)
\end{equation}
and $\Theta_j=(-\rho)\partial\bar{\partial}\psi_j+\partial\bar{\partial}\rho$ is positive
on $\Omega_j$. Let $\chi:{\mathbb R}\rightarrow [0,1]$ be a $C^\infty$ cut-off function such that $\chi|_{(-\infty,-1)}=1$ and $\chi|_{(0,\infty)}=0$. Put
$$
v_j=\bar{\partial}\chi(-\log(-g_j)) \frac{K_\Omega(\cdot,w)}{\sqrt{K_\Omega(w,w)}}.
$$
Let $u_j$ be the solution to $\dbar u_j=v_j$ that is in the orthogonal complement of $\scriptn(\dbar)$ in $L^2(\Omega_j, e^{-\psi_j})$. By Demailly's estimate \eqref{eq:demailly1},
$$
\int_{\Omega_j} |u_j|^2 e^{-\psi_j}\, dV\le \int_{\Omega_j} |v_j|^2_{\partial\bar{\partial}\psi_j}e^{-\psi_j}\, dV.
$$
It follows from \eqref{eq:f1} that the right hand side is uniformly bounded from above, independent of $j$. Passing to a subsequence, we may assume that $u_j$ converges to $u\in L^2(\Omega, e^{-\psi})$ in the weak* topology. (We extend $u_j=0$ on $\Omega\setminus\Omega_j$.) Let
$$
f=\chi(-\log(-g_\Omega(\cdot,w)))K_\Omega(\cdot,w)/K_\Omega(w,w)^{1/2}-u.
$$
Then $f$ is holomorphic on $\Omega$. Since $u$ is holomorphic in a neighborhood of $w$ and $g_\Omega(z,w)=\log|z-w|+O(1)$ near $w$,
\[
u(w)=0.
\]
By Theorem~\ref{prop:bo}, for any $0<r<1$, we have
\begin{align}
(1-r)\int_{\Omega_j}|u_j|^2 (-\rho)^{-r} e^{-\psi_j}\, dV  &\le  \frac1r \int_{\Omega_j} |v_j|^2_{\Theta_j}(-\rho)^{1-r}e^{-\psi_j}\, dV\nonumber\\
& \le \frac{C}r \int_{{\rm supp\,}\bar{\partial}\chi(\cdot)}\frac{|K_\Omega(\cdot,w)|^2}{K_\Omega(w,w)}(-\rho)^{-r}\, dV.
\end{align}
By Theorem~\ref{th:blocki},
\[
\begin{aligned}
{\rm supp\,}\bar{\partial}\chi(\cdot)&\subset\{-e\le g_j(\cdot,w)\le -1\}\subset\{g_\Omega(\cdot, w)\le -1\}\\
&\subset \{C^{-1}|\rho(w)||\log(-\rho(w))|^{-1/a}\le |\rho|\},
\end{aligned}
\]
where $a$ is a Diederich-Forn{\ae}ss exponent for $\Omega$. Therefore, passing to the limit, we have
\begin{equation}\label{eq:lower}
(1-r)\int_\Omega|u|^2 (-\rho)^{-r} e^{-\psi}\, dV\le \frac{C}r\cdot \frac{|\log(-\rho(w))|^{r/a}}{|\rho(w)|^r}.
\end{equation}
Notice that $f$ is a holomorphic function on $\Omega$ such that $f(w)=K_\Omega(w,w)^{1/2}$ and $f=-u$ near $b\Omega$. Since $e^{-\psi}\ge e^{-\varphi}\ge C>0$,  by Lemma~\ref{lm:hardy} (and its proof), $f\in H^2(\Omega)$. Combining with \eqref{eq:lower}, we have
\begin{align*}
\int_{b\Omega} |f|^2\, dS&\le C\lim_{r\to 1^-} (1-r)\int_\Omega|f|^2 (-\rho)^{-r} e^{-\psi}\, dV\\
&=C\lim_{r\to 1^-} (1-r)\int_\Omega|u|^2 (-\rho)^{-r} e^{-\psi}\, dV\\
&\le C \frac{|\log(-\rho(w))|^{1/a}}{|\rho(w)|}.
\end{align*}
It then follows from the extremal property \eqref{eq:szego} of the Szeg\"{o} kernel that
\[
S_\Omega(w,w)\ge C\frac{|\rho(w)|}{|\log(-\rho(w))|^{1/a}}\cdot K_\Omega(w,w).
\]
This concludes the proof of Theorem~\ref{th:main}.

The proof of Theorem~\ref{th:convex} is similar that of Theorem~\ref{th:main}. The only difference is that, instead of Theorem~\ref{th:blocki}, we use the following well known estimate for the pluricomplex Green function on convex domains:
$$
\{g_\Omega(\cdot,w)\le -1\}\subset \{C^{-1}\delta(w)\le \delta(\cdot)\le C\delta(w)\}
$$
(see \cite{Blocki04}, Theorem 5.4).

\begin{remark}
It follows from the Ohsawa-Takegoshi extension theorem \cite{OhsawaTakegoshi87} that for any bounded pseudoconvex Lipschitz domain $\Omega$ in $\C^n$, $K(z, z)\ge C\delta^{-2}(z)$
for some constant $C>0$. Ohsawa  \cite{Ohsawa03} conjectured that the analogue estimate $S(z, z)\ge C\delta^{-1}(z)$ holds. Theorem~\ref{th:convex} confirms this conjecture for convex domains. Moreover, Theorem~\ref{th:main} implies that for a bounded $\delta$-regular domain $\Omega$ with $C^2$-smooth boundary,
\begin{equation}
S(z, z)\ge C\delta^{-1}(z)|\log\delta(z)|^{-1/a},
\end{equation}
where $a$ is any Diederich-Forn{\ae}ss exponent of $\Omega$.
\end{remark}

\bibliography{survey}

\begin{thebibliography}{XX}


\bibitem{Bell86}
S. Bell, \emph{Differential of the Bergman kernel and pseudo-local estimates}, Math. Z. \textbf{192} (1986),  467--472.

\bibitem{Berndtsson01}
B. Berndtsson, {\it Weighted estimates for the $\bar{\partial}$-equation}, Complex Analysis and Complex Geometry (J. D. McNeal eds.),  de Gruyter, pp. 43--57, 2001.

\bibitem{BerndtssonCharpentier00}
B. Berndtsson and Ph. Charpentier, \emph{A Sobolev mapping property of the Bergman knernel}, Math. Z. \textbf{235} (2000), 1-10.

\bibitem{Blocki04}
Z. Blocki, {\it The Bergman metric and the pluricomplex Green function}, Trans. Amer. Math. Soc. {\bf 357} (2004), 2613--2625.

\bibitem{Boas87}
H. P. Boas, {\it The Szeg\"o projection: Sobolev estimates in regular domains}, Trans. Amer. Math. Soc. {\bf 300} (1987), 109--132.

\bibitem{BonamiCharpentier90}
A. Bonami and Ph. Charpentier, \emph{Comparing the Bergman and Szeg\"o projections}, Math. Z. {\bf 204} (1990), 225--233.

\bibitem{BoutetSjostrand76}
L. Boutet de Monvel and J. Sj\"{o}strand, \emph{Sur la singularité des noyaux de Bergman et de Szego}, Journ\'{e}es: Équations aux D\'{e}riv\'{e}es Partielles de Rennes (1975), pp. 123--164. Asterisque, No. 34-35, Soc. Math. France, Paris, 1976.

\bibitem{CSW04}
J. Cao, M.-C. Shaw, and L. Wang, \emph{Estimates for the $\dbar$-Neumann problem
and nonexistence of $C^2$ Levi-flat hypersurfaces in ${\C}P^n$}, Math. Z. \textbf{248} (2004), 183--221.

\bibitem{Catlin84}
D. Catlin, \emph{Global regularity of the {$\overline\partial$}-{Neumann}
  problem}, Complex Analysis of Several Variables (Y.-T. Siu, ed.),
  Proc. Symp. Pure Math., Vol.~41, Amer. Math. Soc., 1984, pp.~39--49.

\bibitem{Catlin87}
\bysame, {\it Subelliptic estimates for the $\bar{\partial}-$Neumann problem on pseudoconvex domains}, Ann. of Math. {\bf 126} (1987), 131--191.

\bibitem{Catlin89}
\bysame, {\it Estimates of invariant metrics on pseudoconvex domains of dimension two}, Math. Z. {\bf 200} (1989), 429--466.

\bibitem{Chen99}
B. Y. Chen, {\it Completeness of the Bergman metric on non-smooth pseudoconvex domains}, Ann. Polon. Math. {\bf 71} (1999), 241--251.

\bibitem{ChenJH89}
J.-H. Chen, {\it Estimates of the invariant metrics on convex domains}, Ph. D. Thesis, Purdue University, 1989.

\bibitem{Demailly82}
J.-P. Demailly, {\it Estimations $L^2$ pour l'op\'{e}rateur $\bar{\partial}$ d'un fibr\'{e} vectoriel holomorphe semi-positif au-dessus d'une vari\'{e}t\'{e} k\"{a}hl\'{e}rienne compl\`{e}te}, Ann. Sci. \'{E}cole Norm. Sup. {\bf 15} (1982), 457--511.

\bibitem{Demailly87}
\bysame, {\it Mesures de Monge-Ampere et mesures plurisousharmoniques}, Math. Z. {\bf 194} (1987), 519--564.

\bibitem{DiederichFornaess77}
K. Diederich and J. E. Forn{\ae}ss, {\it Pseudoconvex domains: bounded strictly plurisubharmonic exhaustion functions}, Invent. Math. {\bf 39} (1977), 129--141.

\bibitem{DiederichFornaess77b}
\bysame, \emph{Pseudoconvex domains: an example with nontrivial Nebenh\"{u}lle}, Math. Ann. \textbf{225} (1977), 275--292.

\bibitem{DFH84}
K. Diederich, J. E. Forn{\ae}ss, and G. Herbort, \emph{Boundary behavior of the Bergman metric}, Complex Analysis of Several Variables (Y.-T. Siu, ed.), Proc. Symp. in Pure Math., Vol.~41, Amer. Math. Soc., 1984, 59--67.

\bibitem{DonnellyFefferman83}
H. Donnelly and Ch. Fefferman, \emph{$L^2$-cohomology and index theorem for the Bergman metric}, Ann. of Math. (2) \textbf{118}(1983), 593--618.

\bibitem{Fefferman74}
C. Fefferman, {\it The Bergman kernel and biholomorphic mappings of pseudoconvex domains}, Invent. Math. {\bf 26} (1974), 1--65.

\bibitem{FornaessHerbig08}
J. E. Forn{\ae}ss and A.-K. Herbig, \emph{A note on plurisubharmonic defining functions in $\C^n$}, Math. Ann. \textbf{342} (2008), 749--772.

\bibitem{GreinerStein77}
P. C. Greiner and E. M. Stein,
{\it Estimate for the $\bar\partial$-Neumann problem}, Princeton University Press, Princeton, New Jersey, 1977.

\bibitem{Harrington07}
P. S. Harrington, {\it The order of plurisubharmonicity on pseudoconvex domains with Lipschitz boundaries}, Math. Res. Lett. {\bf 14} (2007), 485--490.

\bibitem{Herbort99}
G. Herbort, {\it The Bergman metric on hyperconvex domains}, Math. Z. {\bf 232} (1999), 183--196.

\bibitem{Herbort00}
G. Herbort, {\it The pluricomplex Green function on pseudoconvex domains with a smooth boundary}, Inter. J. Math. {\bf 11} (2000), 509--522.

\bibitem{Hormander65}
L. H\"{o}rmander, {\it $L^2-$estimates and existence theorems for the $\bar{\partial}-$equation}, Acta Math. {\bf 113} (1965), 89--152.

\bibitem{Hormander91}
\bysame, \emph{An introduction to complex analysis of several variables}, Elsevier Science Publishers, Amsterdam, The Netherlands, 1991.

\bibitem{Klimek91}
M. Klimek, \emph{Pluripotential theory}, Oxford University Press, 1991.

\bibitem{Kohn02}
J. J. Kohn,  {\it Superlogarithmic estimates on pseudoconvex domains and CR manifolds}, Ann. of Math. {\bf 156} (2002), 213--248.

\bibitem{Machedon88}
M. Machedon, {\it Szeg\"o  kernels on pseudoconvex domains with one degenerate eigenvalue}, Ann. of Math. {\bf 128} (1988), 619--640.

\bibitem{McNeal89}
J. McNeal, {\it Boundary behavior of the Bergman kernel function in $\C^2$}, Duke Math. J. {\bf 58} (1989), 499--512.

\bibitem{McNeal94}
\bysame, {\it Estimates on the Bergman kernels of convex domains}, Adv. Math. {\bf 109} (1994), 108--139.

\bibitem{McNealStein97}
J. D. McNeal and E. M. Stein, {\it The Szeg\"o projection on convex domains}, Math. Z. {\bf 224} (1997), 519--553.

\bibitem{NRSW89}
A. Nagel, J. P. Rosay, E. M. Stein and S. Wainger, {\it Estimates for the Bergman and Szeg\"o kernels in ${\mathbb C}^2$}, Ann. of Math. {\bf 129} (1989), 113--149.

\bibitem{Ohsawa81}
T. Ohsawa, \emph{A remark on the completeness of the Bergman metric}, Proc. Japan Acad., Ser. A, \textbf{57}(1981), 238--240.

\bibitem{Ohsawa03}
T. Ohsawa, Personal communication, 2003.

\bibitem{OhsawaTakegoshi87}
T. Ohsawa and K. Takegoshi, {\it On the extension of $L^2$ holomorphic functions}, Math. Z. {\bf 195} (1987), 197--204.

\bibitem{PhongStein77}
D. H. Phong and E. M. Stein, {\it Estimates for the Bergman and Szeg\"o projections on strongly pseudoconvex domains}, Duke Math. J. {\bf 44} (1977), 695--704.

\bibitem{Richberg68}
R. Richberg, \emph{Stetige streng pseudokonvexe funktionen}, Math. Ann. \textbf{175} (1968), 257--286.

\bibitem{Sibony91}
N. Sibony, {\it Some aspects of weakly pseudoconvex domains}, Several Complex Variables and Complex Geometry (E. Bedford, J. D'Angelo, R. Greene, and S. Krantz, ed.), Proc. Symp. Pure Math., Vol. 52, Amer. Math. Soc., 1991,  199--231.

\bibitem{Siu82}
Y.-T. Siu, \emph{Complex-analyticity of harmonic maps, vanishing and Lefschetz theorems}, J. Differential Geom. \textbf{17} (1982), no. 1, 55--138.

\bibitem{Stein72}
E.~M.~Stein, \emph{Boundary behavior of holomorphic functions of several complex variables}, Princeton University Press, Princeton, New Jersey, 1972.

\bibitem{Stein93}
\bysame, \emph{Harmonic Analysis: Real-Variable Methods, Orthogonality, and Oscillatory Integrals}, Princeton University Press, Princeton, New Jersey, 1993.

\bibitem{Straube97}
 E. J. Straube, \emph{ Plurisubharmonic functions and subelliptic estimates for the $\bar\partial$-Neumann problem on nonsmooth domains}, Math. Res. Letters {\bf 4} (1997), 459--467.
\end{thebibliography}
\providecommand{\bysame}{\leavevmode\hbox
to3em{\hrulefill}\thinspace}

\end{document}